\documentclass[]{theclass} 
\usepackage{graphicx, xfrac, lineno, float, subcaption, tasks, comment, xcolor, booktabs, multirow}
\usepackage[normalem]{ulem}
\usepackage{datetime}
\usepackage{colortbl}
\usepackage{enumerate}

\settasks{
	counter-format=(tsk[r]),
	label-width=4ex
}

\begin{document}
\begin{frontmatter}

\titledata{Ban--Linial's Conjecture and treelike snarks}{To Charlie\ldots for gallantry}  

\authordata{Jean Paul Zerafa}
{Department of Technology and Entrepreneurship Education\\University of Malta, Malta; \\
Department of Computer Science, Faculty of Mathematics, Physics and Informatics\\ Comenius University, Mlynsk\'{a} Dolina, 842 48 Bratislava, Slovakia}
{zerafa.jp@gmail.com}
{The author was partially supported by VEGA 1/0813/18 and VEGA 1/0743/21.}


\begin{abstract}
A bridgeless cubic graph $G$ is said to have a $2$-bisection if there exists a $2$-vertex-colouring of $G$ (not necessarily proper) such that: (i) the colour classes have the same cardinality, and (ii) the monochromatic components are either an isolated vertex or an edge. In 2016, Ban and Linial conjectured that every bridgeless cubic graph, apart from the well-known Petersen graph, admits a $2$-bisection. In the same paper it was shown that every Class I bridgeless cubic graph admits such a bisection. The Class II bridgeless cubic graphs which are critical to many conjectures in graph theory are known as snarks, in particular, those with excessive index at least $5$, that is, whose edge set cannot be covered by four perfect matchings. Moreover, in [\emph{J. Graph Theory}, \textbf{86(2)} (2017), 149--158], Esperet \emph{et al.} state that a possible counterexample to Ban--Linial's Conjecture must have circular flow number at least $5$. The same authors also state that although empirical evidence shows that several graphs obtained from the Petersen graph admit a $2$-bisection, they can offer nothing in the direction of a general proof. Despite some sporadic computational results, until now, no general result about snarks having excessive index and circular flow number both at least $5$ has been proven. In this work we show that treelike snarks, which are an infinite family of snarks heavily depending on the Petersen graph and with both their circular flow number and excessive index at least $5$, admit a $2$-bisection.
\end{abstract}

\end{frontmatter}

\section{Introduction}

All graphs considered in the sequel are finite and simple (without loops and multiple edges). A \emph{bisection} of a cubic graph $G$ is a partition of its vertex set into two disjoint subsets $\mathcal{B}$ and $\mathcal{W}$ such that $|\mathcal{B}|=|\mathcal{W}|$. For simplicity, we shall identify a bisection of $G$ with the $2$-vertex-colouring (not necessarily proper) of $G$, in which every vertex in $\mathcal{B}$ and $\mathcal{W}$ is given colour $1$ (black) and $2$ (white), respectively. In the figures that follow, black and white vertices shall be depicted as filled and unfilled circular vertices, respectively. To distinguish between coloured and uncoloured vertices, the latter shall have a black square shape. A \emph{monochromatic component} is a connected component induced by a colour class of a vertex-colouring. A \emph{$k$-bisection} of a graph $G$ is a bisection of $G$ such that each monochromatic component consists of at most $k$ vertices. Therefore, a \emph{$2$-bisection} is a bisection in which each monochromatic component is a single vertex or an edge. 
A $2$-colouring of the vertices of a graph is said to be \emph{balanced} if the two color classes are of the same cardinality. Otherwise, it is said to be \emph{unbalanced}. 

In this paper we shall consider $2$-bisections in bridgeless cubic graphs, in particular, with respect to the following recent conjecture by Ban and Linial.

\begin{conjecture}[Ban--Linial \cite{banlinial}]\label{banlinial conjecture}
Every bridgeless cubic graph admits a $2$-bisection, except for the Petersen graph. 
\end{conjecture}

We remark that $2$-bisections are equivalent to $4$-weak bisections introduced by Esperet \emph{et al.} in \cite{esperet17}, although, in general, the two definitions do not coincide. In \cite{esperet17}, it was shown that every cubic graph (not necessarily bridgeless) admits a $3$-bisection. Recently this result was extended to the class of simple subcubic graphs by Cui and Liu \cite{cuiliu}, and to the class of subcubic multigraphs by Mattiolo and Mazzuoccolo \cite{mattiolo}. Abreu \emph{et al.} proved Ban--Linial's Conjecture for cycle permutation graphs \cite{abreu}, and for bridgeless claw-free cubic graphs \cite{abreuclaw}. 

In \cite{esperet17}, Esperet \emph{et al.} also showed that a possible counterexample to Ban--Linial's Conjecture must admit circular flow number at least $5$, where $5$ is the largest possible circular flow number according to Tutte's $5$-Flow Conjecture (see \cite{tarsi}). Furthermore, since properly $3$-edge-colourable cubic graphs (Class I cubic graphs) admit a $2$-bisection (see Proposition 7 in \cite{banlinial}), for a possible counterexample to Conjecture \ref{banlinial conjecture} one must search in the class of bridgeless cubic graphs which do not admit a proper $3$-edge-colouring (Class II bridgeless cubic graphs). The graphs in the latter class are the notorious \emph{snarks} which are critical to many conjectures in graph theory (see for example \cite{snarky}). We note that in literature, the word snark is often reserved to Class II bridgeless cubic graphs which are cyclically $4$-edge-connected and have girth at least $5$, however, in what follows, we shall not be specifically making use of this refinement because, as already mentioned in \cite{abreu}, there is no evidence that a minimal counterexample to Ban--Linial's Conjecture is cyclically $4$-edge-connected and has girth at least $5$, if such a counterexample exists. Therefore, in the sequel, snarks shall represent Class II bridgeless cubic graphs.

The least number of perfect matchings needed to cover the edge set of a bridgeless cubic graph $G$ is said to be the \emph{excessive index} of $G$, and the Berge--Fulkerson Conjecture states that every bridgeless cubic graph has excessive index at most $5$ (see \cite{fulkerson,mazzuoccolo}). Class I cubic graphs have excessive index $3$, whilst snarks have excessive index at least $4$. In general, snarks having excessive index exactly $4$ seem to be more manageable than the rest, that is, than those snarks which cannot be covered by four perfect matchings. In fact, two very well-known conjectures---the Cycle Double Cover Conjecture and the \linebreak Fan--Raspaud Conjecture---are true for graphs with excessive index at most $4$ (see \cite{hou, Steffen, ZhangBook} and \cite{fanraspaud,fouquet}, respectively), whilst they are still widely open for the other snarks. Snarks which cannot be covered by four perfect matchings have been the subject of many papers such as \cite{esperetmazzuoccolo,snarky,macajova}. 

As already mentioned before, two infinite families of bridgeless cubic graphs which were shown to satisfy Ban--Linial's Conjecture are cycle permutation graphs and bridgeless claw-free cubic graphs. In \cite{fouquet}, cycle permutation graphs were shown to have excessive index at most $4$, apart from the Petersen graph which has excessive index $5$. On the other hand, it is conjectured \cite{vahan} that bridgeless claw-free cubic graphs have excessive index at most $4$ (see also \cite{521,522}).

Snarks with excessive index $4$ can have circular flow number $5$ (see \cite{goedgebeur}), and so can be critical to Conjecture \ref{banlinial conjecture}. At the same time, snarks with excessive index at least $5$ can have circular flow number strictly less than $5$ \cite{macajova circular}, and consequently admit a $2$-bisection by \cite{esperet17}. However, as far as we know, apart from very few sporadic computational results (see Observation 4.35\footnote[2]{The Petersen graph is the only snark up to $36$ vertices which does not admit a $2$-bisection.} in \cite{abreu}), no snarks having both their circular flow number and excessive index at least $5$ were studied with respect to Ban--Linial's Conjecture. 
Furthermore, Esperet \emph{et al.} \cite{esperet17} stated that several graphs obtained from the Petersen graph were checked with respect to Ban--Linial's Conjecture and they do in fact admit a $2$-bisection (or equivalently, a $4$-weak bisection), however, the authors say that they are far from offering something in the direction of a general proof. In this sense, we think that our main result dealing with treelike snarks, whose circular flow number and excessive index are both at least $5$, is another step towards further providing evidence to the correctness of Conjecture \ref{banlinial conjecture} and surpassing previous encountered hurdles, since treelike snarks heavily depend on the Petersen graph.

\section{Treelike snarks}

Treelike snarks were introduced by Abreu \emph{et al.} \cite{treelike} in an attempt to further enrich the family of snarks which cannot be covered by four perfect matchings. In the same paper, the authors also show that the circular flow number of treelike snarks is at least $5$, and that these snarks admit a $5$-cycle double cover. 

Before proceeding with the definition of these snarks, we introduce multipoles which generalise the notion of graphs. A \emph{multipole} $M$ consists of a set of vertices $V(M)$ and a set of generalised edges such that each generalised edge is either an edge in the usual sense (that is, it has two endvertices) or a dangling edge.  A \emph{dangling edge} is a generalised edge having exactly one endvertex, and a \emph{$k$-pole} is a multipole with $k$ dangling edges.
The set of dangling edges of $M$ is denoted by $\partial M$ whilst the set of edges of $M$ having two endvertices is denoted by $E(M)$. 
Two dangling edges are \emph{joined} if they are both deleted and their endvertices are made adjacent. In a similar way, a dangling edge with endvertex $x$ is said to be \emph{joined} to a vertex $y$, if the dangling edge is deleted and $x$ and $y$ are made adjacent.

Let $s$ and $t$ be two adjacent vertices of the Petersen graph $P$ and let $P'$ be the graph obtained after deleting $s$ and $t$ (and all the edges incident to them) from $P$.
Consider the $4$-pole $M$ obtained from $P'$ such that $V(M)=V(P')$, $E(M)=E(P')$, and each vertex of degree $2$ in $P'$ has exactly one dangling edge incident to it in $M$. We define the \emph{left dangling edges} to be the ones corresponding to the edges originally incident to $s$ and not $t$, and the \emph{right dangling edges} to be the ones corresponding to the edges originally incident to $t$ and not $s$. We also partition the four dangling edges in ordered pairs, say $(l_1,l_2)$, referred to as the first and second left dangling edges, and $(r_1,r_2)$, referred to as the first and second right dangling edges. Due to the symmetry of the Petersen graph the order of the dangling edges in each set of the partition is not relevant, but for simplicity and consistency, we shall assume that $(l_1,l_2)$ and $(r_1,r_2)$ look as in Figure \ref{figure petersen 4pole}. This $4$-pole is said to be a \emph{Petersen $4$-pole}.

\begin{figure}[h]
      \centering
      \includegraphics[width=0.65\textwidth]{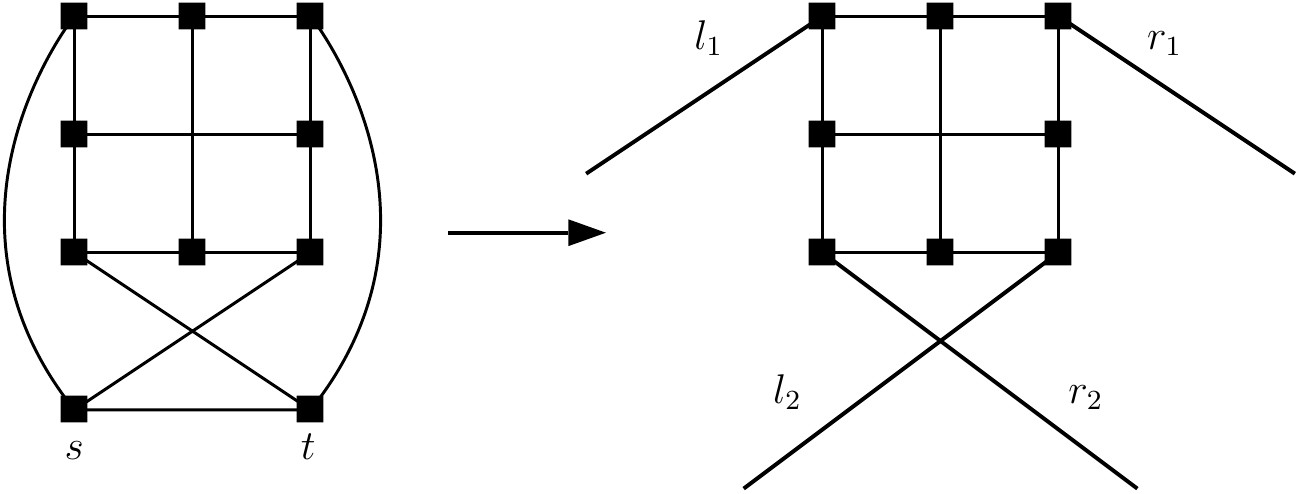}
      \caption{The Petersen graph and a Petersen $4$-pole}
      \label{figure petersen 4pole}
\end{figure}

A \emph{Halin graph} is a plane graph consisting of a planar representation of a tree without degree $2$ vertices, and a circuit on the set of its leaves (see \cite{Hal64}). We remark that leaves are the vertices in a tree having degree $1$.
Let $H$ be a cubic Halin graph consisting of the tree $T$ and the circuit $K$. A \emph{treelike snark} $G$ is the bridgeless cubic graph that can be obtained by the following procedure:
\begin{itemize}
\item for every leaf $x$ of $T$, we add two new vertices, say $x_1$ and $x_2$, and the edges $xx_1$ and $xx_2$; and
\item for every edge $xy$ of $K$, with $x$ being the predecessor of $y$ with respect to the clockwise orientation of $K$, the edge $xy$ is replaced with a Petersen $4$-pole, such that the first and second left dangling edges of this Petersen $4$-pole are joined to $x_1$ and $x_2$, respectively, whilst the first and second right dangling edges are joined to $y_1$ and $y_2$, respectively.
\end{itemize}

Two leaves $x$ and $y$ of $T$ are called \emph{consecutive} if they are adjacent in the circuit $K$, and we shall say that the Petersen $4$-pole of $G$ replacing the edge $xy$ of $K$ is \emph{in between} the two leaves $x$ and $y$.
Moreover, two consecutive leaves are said to be \emph{near} if they have distance two in $T$, that is, they have a common neighbour in $T$. We remark that $T$ always has two near leaves. Similarly, two Petersen $4$-poles $A$ and $B$ are called \emph{consecutive} if there exist three consecutive leaves $x,y,z$ (that is, $x$ and $y$ are consecutive and $y$ and $z$ are consecutive) such that $A$ is in between $x$ and $y$, and $B$ is in between $y$ and $z$. Again, we say that the leaf $y$ is \emph{in between} the Petersen $4$-poles $A$ and $B$.

We next consider various $2$-vertex-colourings of a Petersen $4$-pole, not necessarily balanced, in which each monochromatic component consists of at most two vertices. For simplicity, the latter shall be referred to as the \emph{monochromatic property}. A $2$-vertex-colouring of the Petersen $4$-pole is said to be \emph{all $1$-balanced} if only the four endvertices of the dangling edges are coloured $1$ (see Figure \ref{figure type a}). For short, a Petersen $4$-pole given this colouring is said to be an all $1$-balanced Petersen $4$-pole. An all $2$-balanced Petersen $4$-pole is defined in a similar way by interchanging the colours $1$ and $2$ in the $4$-pole.

\begin{figure}[h]
      \centering
      \includegraphics[width=0.75\textwidth]{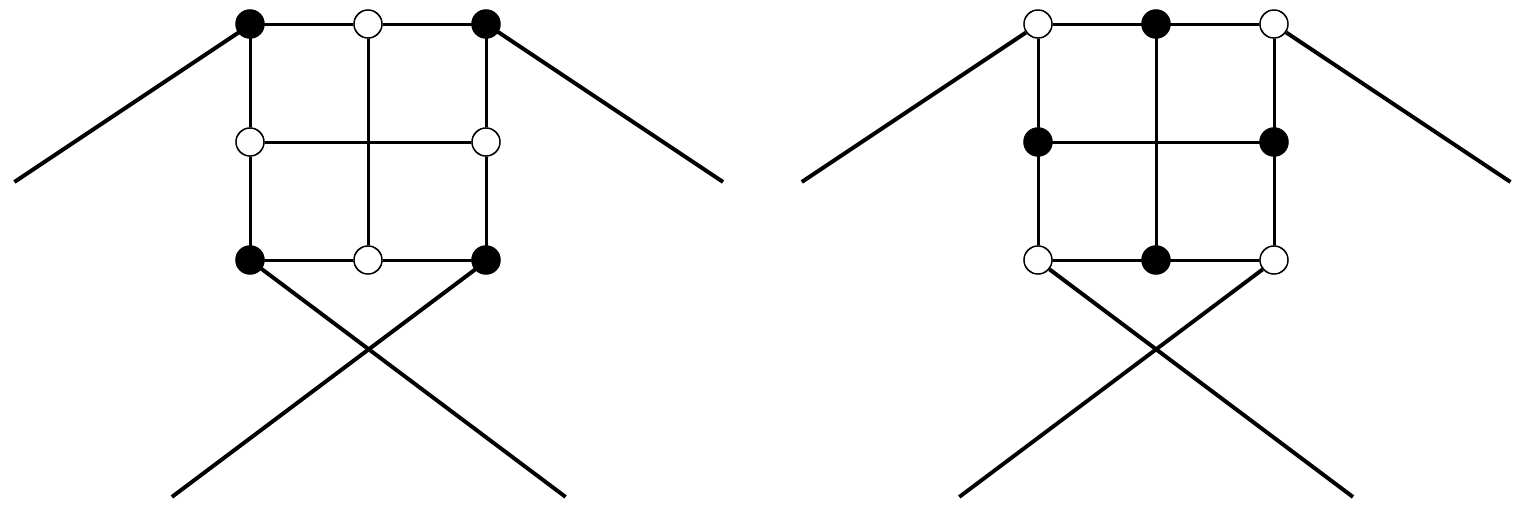}
      \caption{An all $1$-balanced and an all $2$-balanced Petersen $4$-pole}
      \label{figure type a}
\end{figure}

We also consider a balanced $2$-vertex-colouring of a Petersen $4$-pole (respecting the monochromatic property) in which exactly two endvertices of the four dangling edges are coloured $1$. Such a colouring is said to be \emph{$(i,j)$-balanced}, where $i$ and $j$ belong to $\{1,2\}$ and are equal to the colour of the endvertices of the first left and first right dangling edges, respectively (see Figure \ref{figure type b}). A Petersen $4$-pole having this colouring is said to be an $(i,j)$-balanced Petersen $4$-pole. When $i$ and $j$ are equal, say to $1$, we distinguish between the \emph{left} and \emph{right} $(1,1)$-balanced Petersen $4$-poles as follows: the left (similarly right) $(1,1)$-balanced colouring corresponds to the one in which the endvertex of the first left (similarly right) dangling edge has a neighbour in the Petersen $4$-pole which is coloured $1$.

\begin{figure}[h]
      \centering
      \includegraphics[width=0.75\textwidth]{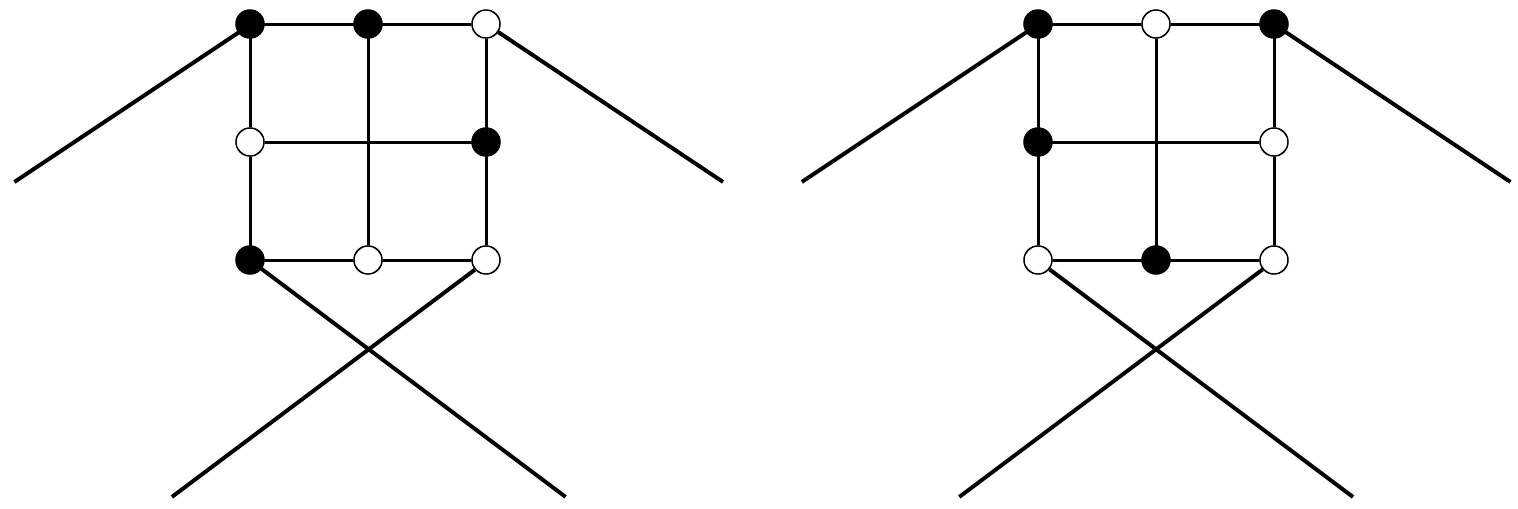}
      \caption{A $(1,2)$-balanced and a left $(1,1)$-balanced Petersen $4$-pole}
      \label{figure type b}
\end{figure}

Finally, we consider a $2$-vertex-colouring of a Petersen $4$-pole which once again respects the monochromatic property but is now unbalanced, that is, the two colour classes do not have the same cardinality. In this colouring $|\mathcal{B}|=|\mathcal{W}|+2$ (or vice-versa) and exactly three of the endvertices of the dangling edges of a Petersen $4$-pole are given the same colour, say $1$, without loss of generality (see Figure \ref{figure type c}). If, for example, the endvertex of the first left (similarly right) dangling edge is the one having colour $2$, the colouring and the Petersen $4$-pole are said to be \emph{{\scriptsize $^{2}$}$1$-unbalanced} (similarly $1${\scriptsize $^{2}$}-unbalanced). The {\scriptsize $_{2}$}$1$- and $1${\scriptsize $_{2}$}-unbalanced Petersen $4$-poles are defined in a similar manner.

\begin{figure}[h]
      \centering
      \includegraphics[width=0.75\textwidth]{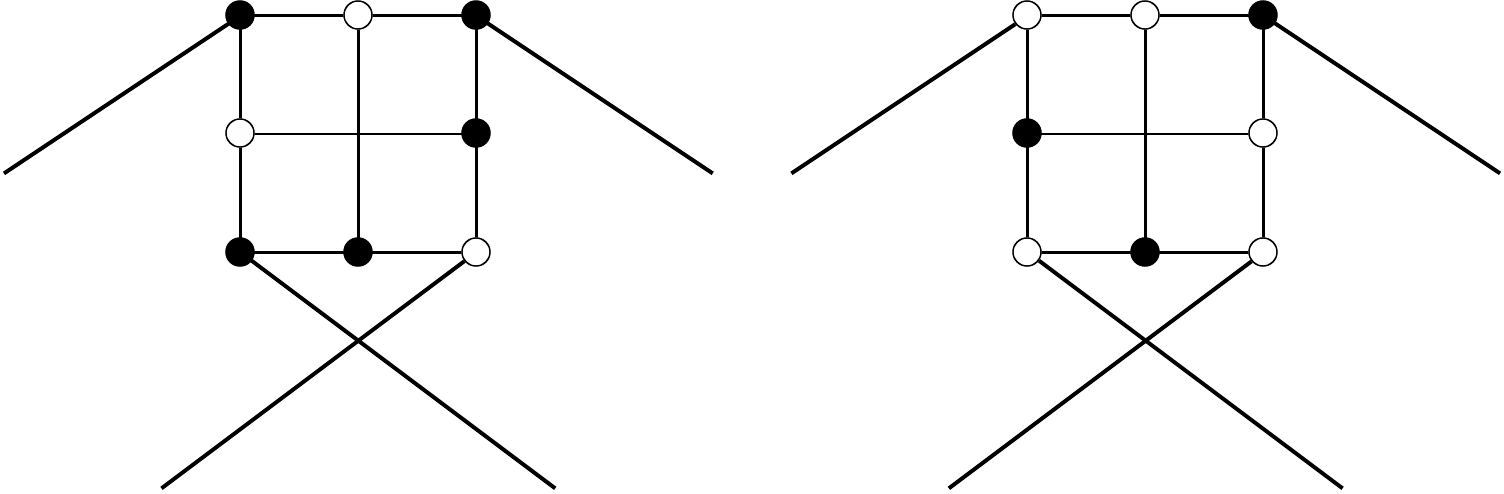}
      \caption{A {\scriptsize $_{2}$}$1$-unbalanced and a $2${\scriptsize $^{1}$}-unbalanced Petersen $4$-pole}
      \label{figure type c}
\end{figure}

One can clearly see that the above constitute (up to symmetry) all the possible $2$-vertex-colourings of a Petersen $4$-pole in which each monochromatic component has at most two vertices.

\subsection{Main result}
\begin{theorem}\label{main theorem}
Every treelike snark admits a $2$-bisection.
\end{theorem}

\begin{proof}
Let $G$ be a treelike snark. Suppose that $G$ does not admit a $2$-bisection. We assume that the defining tree $T$ of $G$ is of minimum order. Observation 4.35 in \cite{abreu} states that the Petersen graph is the only snark up to $36$ vertices which does not admit a $2$-bisection. Consequently, the smallest treelike snark on $34$ vertices has a $2$-bisection (see also Figure \ref{figure windmill}), and so we can assume that $T$ has at least two degree $3$ vertices.
\begin{figure}[h]
      \centering
      \includegraphics[width=0.75\textwidth]{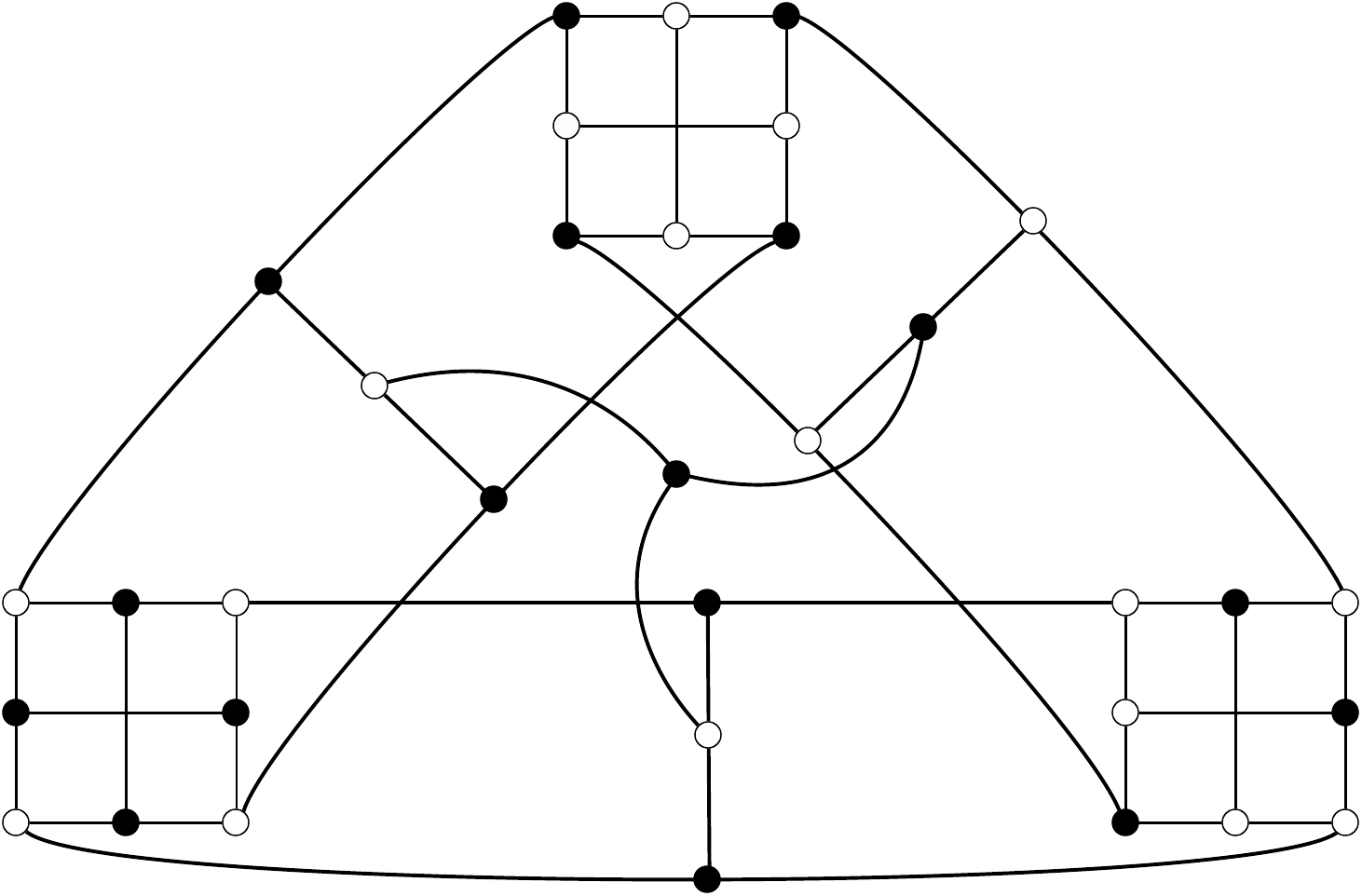}
      \caption{A $2$-bisection of the smallest treelike snark}
      \label{figure windmill}
\end{figure}
As depicted in Figure \ref{figure induction}, consider two near leaves of $T$, say $x$ and $y$, and let $e$ and $f$ be the two edges of $T$ incident to $x$ and $y$, respectively. Moreover, let $g$ be the edge of $T$ adjacent to $e$ and $f$. Consider the three consecutive Petersen $4$-poles $A,B,C$ such that $x$ is in between $A$ and $B$, and $y$ is in between $B$ and $C$. Let $a_{1}$ and $a_{2}$ be the two endvertices of the first and second right dangling edges of the Petersen $4$-pole $A$, respectively, and let $c_{1}$ and $c_{2}$ be the two endvertices of the first and second left dangling edges of the Petersen $4$-pole $C$, respectively. Let $x_{1}$ and $x_{2}$ be the two neighbours of $x$ in $G$ adjacent to $a_{1}$ and $a_{2}$, respectively. Let $y_{1}$ and $y_{2}$ be the two neighbours of $y$ in $G$ adjacent to $c_{1}$ and $c_{2}$, respectively. Let $u$ be the vertex in $T$ mutually adjacent to $x$ and $y$, and let $v$ be the other neighbour of $u$ in $T$.

We construct a smaller treelike snark $G'$ following the procedure presented in Figure \ref{figure induction}. Since $T$ has at least two vertices of degree $3$, the resulting graph $G'$ is indeed another treelike snark. Let $T'$ be the tree defining $G'$. As in Figure \ref{figure induction}, $z$ is the leaf in between the Petersen $4$-poles $A$ and $C$ in $G'$. In this case, let $z_{1}$ and $z_{2}$ be the common neighbours of $a_{1}$ and $c_{1}$, and $a_{2}$ and $c_{2}$, respectively.

\begin{figure}[h]
\centering
      \includegraphics[width=1\textwidth]{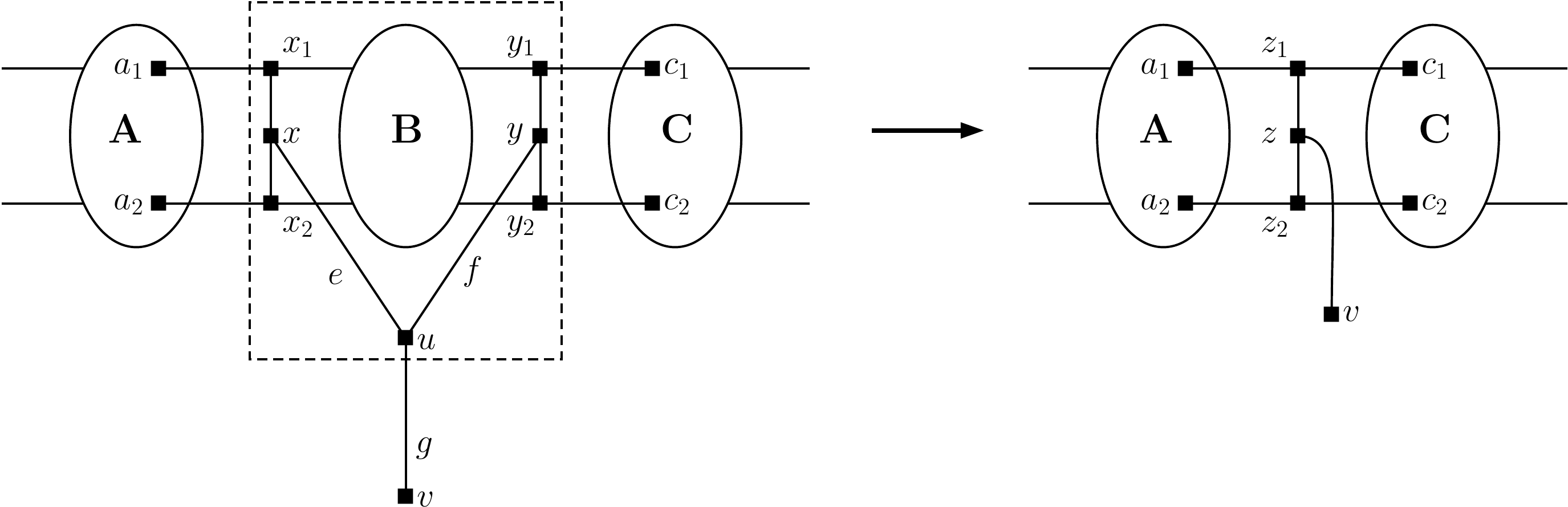}
\caption{Induction step in Theorem \ref{main theorem}}
\label{figure induction}
\end{figure}

Since $|V(T')|<|V(T)|$, $G'$ admits a $2$-bisection. In what follows we show that we can always extend the $2$-vertex-colouring corresponding to a $2$-bisection of $G'$, to a $2$-vertex-colouring of $G$ which shall give rise to a $2$-bisection of the bigger treelike snark $G$, a contradiction to our initial assumption. We consider five cases, depending on the colours of $a_{1},a_{2},c_{1},c_{2}$ in $G'$.\\

\noindent\textbf{Case I.} The colours of $a_{1},a_{2},c_{1},c_{2}$ in $G'$ are all the same. 

Without loss of generality, assume that the colours of $a_{1},a_{2},c_{1},c_{2}$ are all $1$. Consequently, the colours of $(z_{1},z,z_{2})$ are $(2,1,2)$. The colour given to $v$ can be either $1$ or $2$. Suppose first that the colour of $v$ is $1$. We colour $G$ as follows. The vertices in $G$ corresponding to $V(G')-\{z_{1},z,z_{2}\}$ are assigned the same colour given to them in $G'$, and the vertices $(x_{1},x,x_{2})$ and $(y_{1},y,y_{2})$ are both assigned the colours $(2,1,2)$, that is, the colours given to $(z_{1},z,z_{2})$ in $G'$. The vertex $u$ is given colour $2$ and consequently, this partial $2$-vertex-colouring of $G$, that is, with the vertices of $B$ still uncoloured, respects the monochromatic property, and, so far, there are two more vertices coloured $2$. However, since the colours of $a_{1},a_{2},c_{1},c_{2}$ are all $1$, and the colours of $x_{1},x_{2},y_{1},y_{2}$ are all $2$, we can let $B$ be a ${1}${\scriptsize $_{2}$}-unbalanced Petersen $4$-pole, giving a $2$-bisection of $G$, a contradiction to our initial assumption.
Therefore, $v$ must be coloured $2$ in $G'$. The vertices in $G$ corresponding to $V(G')-\{z_{1},z,z_{2}\}$ are assigned the same colour given to them in $G'$, and the vertices $(x_{1},x,x_{2})$ and $(y_{1},y,y_{2})$ are assigned the colours $(2,2,1)$ and $(2,1,2)$, respectively. The vertex $u$ is coloured $1$ as $z$. Consequently, this partial $2$-vertex-colouring is balanced. If $a_{2}$ does not have a neighbour in $A$ which is coloured $1$, then every monochromatic component in this partial $2$-vertex-colouring has order at most two. In this case, if we colour $B$ as a left $(1,1)$-balanced Petersen $4$-pole, a $2$-bisection of $G$ is obtained, a contradiction. Therefore, $a_{2}$ must have a neighbour in $A$ coloured $1$. Since $a_{1}$ is coloured $1$, $A$ is either a {\scriptsize $^{2}$}$1$- or a {\scriptsize $_{2}$}$1$-unbalanced Petersen $4$-pole. Consequently, if we change the colour of $a_{2}$ from $1$ to $2$, $A$ becomes a $(2,1)$- or a right $(1,1)$-balanced Petersen $4$-pole, giving a partial $2$-vertex-colouring of $G$ respecting the monochromatic property, with the vertices coloured $2$ now being two more than those coloured $1$. However, colouring $B$ as a {\scriptsize $_{2}$}$1$-unbalanced Petersen $4$-pole, a $2$-bisection of $G$ is obtained, a contradiction. Thus, Case I cannot occur.\\

\textbf{Case II.} In $G'$, the vertices $a_{1},a_{2}$ are coloured $1$ and the vertices $c_{1},c_{2}$ are coloured $2$, or vice-versa.

Without loss of generality, assume that the colours of $a_{1},a_{2}$ are $1$, and the colours of $c_{1},c_{2}$ are $2$. We first claim that the colours of $z_{1}$ and $z_{2}$ must be the same. For, suppose not. If $z_{1}$ is coloured $1$, then $z$ must be coloured $2$, but this implies that $z,z_{2},c_{2}$ are all coloured $2$, a contradiction. Thus, without loss of generality, we assume that $z_{1},z_{2}$ are coloured $2$, implying that $z$ is coloured $1$. We consider two cases, depending on whether $v$ is coloured $1$ or $2$. Suppose first that $v$ is coloured $1$. We colour $G$ as follows. The vertices in $G$ corresponding to $V(G')-\{z_{1},z,z_{2}\}$ are assigned the same colour given to them in $G'$, and the vertices $(x_{1},x,x_{2})$ and $(y_{1},y,y_{2})$ are both assigned the colours $(2,1,2)$. The vertex $u$ is given colour $2$ and consequently, every monochromatic component in this partial $2$-vertex-colouring has order at most two. However, there are two more vertices coloured $2$ so far. Since the colours of $a_{1},a_{2}$ are both $1$, and the colours of $x_{1},x_{2}$ are both $2$, we can let $B$ either be a {\scriptsize $_{2}$}$1$- or a {\scriptsize $^{2}$}$1$-unbalanced Petersen $4$-pole, giving a $2$-bisection of $G$. This is a contradiction, and so $v$ must be coloured $2$ in $G'$. As before, we proceed by colouring the vertices in $G$ corresponding to $V(G')-\{z_{1},z,z_{2}\}$ with the same colour given to them in $G'$. The vertices $(x_{1},x,x_{2})$ are assigned the colours $(2,1,2)$, and since $c_{1},c_{2}$ are both coloured $2$, we can colour $(y_{1},y,y_{2})$ with the colours $(1,2,1)$. The vertex $u$ is coloured $1$ as $z$. One can see that all the monochromatic components in this partial $2$-vertex-colouring have order at most two, but, so far, there are two more vertices coloured $1$. If we let $B$ be a {\scriptsize $_{1}$}$2$-unbalanced Petersen $4$-pole, we obtain a $2$-bisection of $G$, a contradiction once again. Therefore, we cannot have Case II.\\

\textbf{Case III.} In $G'$, the vertices $a_{1},c_{1}$ are coloured $1$ and the vertices $a_{2},c_{2}$ are coloured $2$, or vice-versa.

Without loss of generality, assume that the colours of $a_{1},c_{1}$ are $1$, and the colours of $a_{2},c_{2}$ are $2$. In this case, $(z_{1},z_{2})$ must be coloured $(2,1)$. Without loss of generality, assume that $v$ is coloured $1$. Therefore, $z$ must be coloured $2$. We colour $G$ as follows. The vertices in $G$ corresponding to $V(G')-\{z_{1},z,z_{2}\}$ are assigned the same colour given to them in $G'$, and the vertices $(x_{1},x,x_{2})$ and $(y_{1},y,y_{2})$ are assigned the colours $(1,2,1)$ and $(2,1,1)$, respectively. The vertex $u$ is given colour $2$ as $z$. If the vertex $a_{1}$ has no neighbour in $A$ coloured $1$, then this partial $2$-vertex-colouring respects the monochromatic property and, so far, there are two more vertices coloured $1$. However, letting $B$ be a $2${\scriptsize $^{1}$}-unbalanced Petersen $4$-pole gives a $2$-bisection of $G$,  a contradiction. Therefore, $a_{1}$ has a neighbour in $A$ coloured $1$. Therefore, $A$ must be a right $(1,1)$- or a $(2,1)$-balanced Petersen $4$-pole. If we change the colour $1$ of $a_{1}$ to $2$, $A$ becomes a {\scriptsize $^{1}$}$2$- or a {\scriptsize $_{1}$}$2$-unbalanced Petersen $4$-pole. Because of this step, the resulting partial $2$-vertex-colouring of $G$ is now balanced and respects the monochromatic property. However, if we let $B$ to be the all $2$-balanced Petersen $4$-pole, we get that $G$ has a $2$-bisection, a contradiction. Hence, we cannot have Case III.\\

\textbf{Case IV.} In $G'$, the vertices $a_{1},c_{2}$ are coloured $1$ and the vertices $a_{2},c_{1}$ are coloured $2$, or vice-versa.

Without loss of generality, assume that the colours of $a_{1},c_{2}$ are $1$, and the colours of $a_{2},c_{1}$ are $2$. As in Case II, $z_{1}$ and $z_{2}$ must have the same colour, and so, without loss of generality, we can assume that these two vertices are both coloured $2$. Consequently, $z$ must be coloured $1$. We consider two cases depending on whether $v$ is coloured $1$ or $2$. Suppose first that the colour of $v$ is $1$. We extend the colouring of $G'$ to $G$ as follows. The vertices in $G$ corresponding to $V(G')-\{z_{1},z,z_{2}\}$ are assigned the same colour given to them in $G'$, and the vertices $(x_{1},x,x_{2})$ and $(y_{1},y,y_{2})$ are both assigned the colours $(2,1,2)$. The vertex $u$ is given colour $2$ and consequently, every monochromatic component in this partial $2$-vertex-colouring has order at most two, and, so far, there are two more vertices coloured $2$. However, letting $B$ be a {\scriptsize $^{2}$}$1$-unbalanced Petersen $4$-pole gives a $2$-bisection of $G$, a contradiction. Therefore, $v$ must be coloured $2$. The vertices in $G$ corresponding to $V(G')-\{z_{1},z,z_{2}\}$ are assigned the same colour given to them in $G'$, but this time the vertices $(x_{1},x,x_{2})$ and $(y_{1},y,y_{2})$ are assigned the colours $(2,2,1)$ and $(2,1,2)$, respectively. The vertex $u$ is coloured $1$ as $z$. As a result, we can see that this partial $2$-vertex-colouring is balanced and respects the monochromatic property. However, letting $B$ be an all $1$-balanced Petersen $4$-pole gives a $2$-bisection of $G$, a contradiction once again. So we must have Case V.\\

\textbf{Case V.} Exactly three of $a_{1},a_{2},c_{1},c_{2}$ have the same colour in $G'$.

Without loss of generality, assume that $a_{1},a_{2},c_{1}$ have colour $1$ and $c_{2}$ has colour $2$. Consequently, the colour of $z_{1}$ is $2$. We consider three cases depending on the colours of $(v,z,z_{2})$. When the colour of $v$ is $1$, $(z,z_{2})$ can be coloured $(1,2)$ or $(2,1)$. Suppose we have the former case. We colour $G$ as follows. The vertices in $G$ corresponding to $V(G')-\{z_{1},z,z_{2}\}$ are assigned the same colour given to them in $G'$, and the vertices $(x_{1},x,x_{2})$ and $(y_{1},y,y_{2})$ are both assigned the colours $(2,1,2)$. The vertex $u$ is given colour $2$ and consequently, the resulting partial $2$-vertex-colouring respects the monochromatic property and, so far, there are two more vertices coloured $2$. However, if we let $B$ be a $1${\scriptsize $^{2}$}-unbalanced Petersen $4$-pole, a $2$-bisection of $G$ is obtained, a contradiction.
Therefore, the colours of $(z,z_{2})$ are $(2,1)$. We colour $G$ as follows. The vertices in $G$ corresponding to $V(G')-\{z_{1},z,z_{2}\}$ are assigned the same colour given to them in $G'$, but this time the vertices $(x_{1},x,x_{2})$ and $(y_{1},y,y_{2})$ are assigned the colours $(2,1,2)$ and $(2,1,1)$, respectively. The vertex $u$ is coloured $2$ as $z$, and consequently, the resulting partial $2$-vertex-colouring is balanced and respects the monochromatic property. However, if we let $B$ be a right $(1,1)$-balanced Petersen $4$-pole, a $2$-bisection of $G$ is obtained, a contradiction. 
Hence, we must have that the colours of $(v,z,z_{2})$ are $(2,1,2)$. The vertices in $G$ corresponding to $V(G')-\{z_{1},z,z_{2}\}$ are assigned the same colour given to them in $G'$, and the vertices $(x_{1},x,x_{2})$ and $(y_{1},y,y_{2})$ are assigned the colours $(2,1,2)$ and $(2,2,1)$, respectively. The vertex $u$ is given colour $1$ as $z$, and consequently, this partial $2$-vertex-colouring is balanced and every monochromatic component has order at most two. However, if we let $B$ be an all-1 balanced Petersen $4$-pole, a $2$-bisection of $G$ is obtained, a contradiction once again. \\

Since there are no more cases to consider, our initial assumption is wrong. Therefore, every treelike snark admits a $2$-bisection, proving our theorem.
\end{proof}

Finally, we remark that it would be quite intriguing to see whether the above proof can be further extended to accommodate a recent class of snarks which contains treelike snarks: \emph{Halin snarks}. These snarks, introduced in \cite{macajova} by M\'a\v{c}ajov\'a and \v{S}koviera, were shown to have excessive index at least $5$. Moreover, whenever the building blocks (referred to as Halin fragments) used for the construction of a Halin snark are exclusively made up of decollineators (see \cite{macajova} for definition) originating from snarks having circular flow number at least $5$, the circular flow number of the resulting Halin snark is at least $5$ as well. In the above proof, we explicitly took advantage of the symmetry of Petersen $4$-poles and of the behaviour of $2$-vertex-colourings (and their corresponding monochromatic components) in such a pole. In this sense, we remark that in general it is still not clear what assumptions need to be made with regards to the building blocks of Halin snarks in order to tackle this more general problem.


\begin{thebibliography}{99}

\bibitem{abreuclaw} M. Abreu, J. Goedgebeur, D. Labbate and G. Mazzuoccolo, A note on $2$-bisections of claw-free cubic graphs, \emph{Discrete Appl. Math.} \textbf{244} (2018), 214--217.

\bibitem{abreu} M. Abreu, J. Goedgebeur, D. Labbate and G. Mazzuoccolo, Colourings of cubic graphs inducing isomorphic monochromatic subgraphs, \emph{J. Graph Theory}, \textbf{92} (2019), 415--444.

\bibitem{treelike}
M. Abreu, T. Kaiser, D. Labbate and G. Mazzuoccolo,
Treelike snarks,
\emph{Electron. J. Combin.} \textbf{23(3)} (2016), $\#$P3.54.

\bibitem{banlinial}A. Ban and N. Linial, Internal partitions of regular graphs, \emph{J. Graph Theory}, \textbf{83(1)} (2016), 5--18.

\bibitem{521}U.A. Celmins, \emph{On Cubic Graphs That Do Not Have an Edge-3-Colouring}, Ph.D. Thesis, Department of Combinatorics and Optimization, University of Waterloo, Waterloo, Canada, 1984.

\bibitem{cuiliu}Q. Cui and W. Liu, A note on $3$-bisections in subcubic graphs, \emph{Discrete Appl. Math.} \textbf{285} (2020), 147--152.

\bibitem{esperetmazzuoccolo}
L. Esperet and G. Mazzuoccolo,
On cubic bridgeless graphs whose edge-set cannot be covered by four perfect matchings,
\emph{J. Graph Theory} \textbf{77(2)} (2014), 144--157.

\bibitem{esperet17} L. Esperet, G. Mazzuoccolo and M. Tarsi, Flows and Bisections in Cubic Graphs, \emph{J. Graph Theory}, \textbf{86(2)} (2017), 149--158.

\bibitem{fanraspaud}
G. Fan and A. Raspaud,
Fulkerson's Conjecture and circuit covers,
\emph{J. Combin. Theory Ser. B} \textbf{61(1)} (1994), 133--138.

\bibitem{fouquet} J.-L. Fouquet and J.-M. Vanherpe, \emph{On the perfect matching index of bridgeless cubic graphs}, 2009, preprint, \href{https://arxiv.org/abs/0904.1296}{arXiv:0904.1296}.

\bibitem{fulkerson}
D.R. Fulkerson,
Blocking and anti-blocking pairs of polyhedra,
\emph{Math. Program.} \textbf{1(1)} (1971), 168--194.

\bibitem{tarsi}
L.A. Goddyn,  M. Tarsi and C.-Q. Zhang,
On $(k,d)$-colorings and fractional nowhere-zero flows,
\emph{J. Graph Theory} \textbf{28(3)} (1998), 155--161.

\bibitem{goedgebeur}
J. Goedgebeur, D. Mattiolo and G. Mazzuoccolo,
Computational results and new bounds for the circular flow number of snarks,
\emph{Discrete Math.} \textbf{343(10)} (2020), 112026.

\bibitem{vahan}
A. Hakobyan and V. Mkrtchyan,
$S_{12}$ and $P_{12}$-colorings of cubic graphs,
\emph{Ars Math. Contemp.} \textbf{17} (2019), 431--445.

\bibitem{Hal64} 
R. Halin,
\"{U}ber simpliziale Zerf\"{a}llungen beliebiger (endlicher oder unendlicher) Graphen,
{\em Math. Ann.} \textbf{156} (1964), 216--225.

\bibitem{hou}
X. Hou, H.J. Lai and C.-Q. Zhang,
On Perfect Matching Coverings and Even Subgraph Coverings,
\emph{J. Graph Theory} \textbf{81} (2016), 83--91.

\bibitem{snarky}E. M\'{a}\v{c}ajov\'{a}, G. Mazzuoccolo, V. Mkrtchyan and J.P. Zerafa,
\emph{Some snarks are worse than others}, \emph{European J. Combin.}, to appear, \href{https://arxiv.org/abs/2004.14049}{arXiv:2004.14049}.

\bibitem{macajova}
E. M\'a\v{c}ajov\'a and M. \v{S}koviera,
Cubic graphs that cannot be covered with four perfect matchings,
\emph{J. Combin. Theory Ser. B} \textbf{150} (2021), 144--176. 

\bibitem{macajova circular}
E. M\'a\v{c}ajov\'a and M. \v{S}koviera,
Perfect Matching Index versus Circular Flow Number of a Cubic Graph,
\emph{SIAM J. Discrete Math.} \textbf{35(2)} (2021), 1287--1297.

\bibitem{mattiolo}
D. Mattiolo and G. Mazzuoccolo, On $3$-Bisections in Cubic and Subcubic Graphs, \emph{Graphs and Combinatorics} \textbf{37} (2021), 743--746. \href{https://doi.org/10.1007/s00373-021-02275-z}{https://doi.org/10.1007/s00373-021-02275-z}.

\bibitem{mazzuoccolo}
G. Mazzuoccolo,
The  equivalence  of  two  conjectures  of  Berge  and  Fulkerson,
\emph{J. Graph Theory} \textbf{68} (2011), 125--128.

\bibitem{522} M. Preissmann, \emph{Sur les colorations des ar\^{e}tes des graphes cubiques}, Ph.D. thesis, Universit\'{e} Joseph Fourier, Grenoble, France, 1981, \href{https://tel.archives-ouvertes.fr/tel-00294175}{https://tel.archives-ouvertes.fr/tel-00294175}.

\bibitem{Steffen}
E. Steffen,
1-Factor and Cycle Covers of Cubic Graphs,
\emph{J. Graph Theory} \textbf{78(3)} (2015), 195--206.

\bibitem{ZhangBook}
C.-Q. Zhang,
\emph{Integer Flows and Cycle Covers of Graphs},
first ed., Marcel Dekker, New York, 1997.

\end{thebibliography}
\end{document}